\documentclass[leqno]{amsart}
\usepackage{amsfonts,amssymb,amsmath,amsgen,amsthm}
\usepackage{amsmath}
\usepackage{amsfonts}
\usepackage{amssymb}

\usepackage{hyperref}
\usepackage{color}
\newcommand{\msc}[2][2000]{%
  \let\@oldtitle\@title%
  \gdef\@title{\@oldtitle\footnotetext{#1 \emph{Mathematics subject
        classification.} #2}}% 
}

\theoremstyle{plain} \newtheorem{theorem}{Theorem} [section]

\newtheorem{lemma}[theorem]{Lemma}

\newtheorem{proposition}[theorem]{Proposition}
 \theoremstyle{remark}

\def\R{{\mathbb R}}% real numbers
\def\Sch{{\mathcal S}}% Schwartz space 
\def\Sph{{\mathbb S}}% sphere
\def\O{\mathcal O}
\def\F{\mathcal F}  

\def\({\left(}
\def\){\right)}
\def\<{\left\langle}
\def\>{\right\rangle}

\def\d{{\partial}} \def\eps{\varepsilon} 
\def\si{{\sigma}}  

\DeclareMathOperator{\RE}{Re} \DeclareMathOperator{\IM}{Im}

\numberwithin{equation}{section}

\begin{document}

\title[Logarithmic fractional Schr\"odinger equation]{On the Cauchy problem for logarithmic fractional Schr\"odinger equation}
\author[R. Carles]{R\'emi Carles}
\author[F. Dong]{Fangyuan Dong}
\address{Univ Rennes, CNRS\\ IRMAR - UMR
  6625\\ F-35000 Rennes, France}
\email{Remi.Carles@math.cnrs.fr}
\email{dfangyuan20@163.com}

\begin{abstract}
  We consider the fractional Schr\"odinger equation with a logarithmic
  nonlinearity, when the power of the Laplacian is between zero and
  one. We prove global existence results in three different functional
  spaces: the Sobolev space corresponding to the quadratic form domain
  of the fractional Laplacian, the energy space, and a space contained
  in the operator domain of the fractional Laplacian. For this last
  case, a finite momentum assumption is made, and the key step
  consists in estimating the Lie commutator between the fractional
  Laplacian and the multiplication by a monomial. 
\end{abstract}
\thanks{Partially supported by Centre Henri Lebesgue, program
ANR-11-LABX-0020-0. The second author is on leave from the University
of Science and Technology in Beijing, thanks to
some funding from the Chinese Scholarship Council. A CC-BY public
copyright license has been applied by the authors to the present
document and will be applied to all subsequent versions up to the
Author Accepted Manuscript arising from this submission. }  
\maketitle

\section{Introduction} 
We consider the logarithmic Schr\"odinger equation
\begin{equation}\label{eq:logNLS} 
i\d_t u -(-\Delta)^s u =\lambda \log\(|u|^2\)u\, ,\quad u_{\mid t=0} =u_0 \, , 
\end{equation} 
where $0 < s < 1$, $u = u(t,x)$ represents a complex-valued function defined on $(t,x) \in \mathbb{R} \times \mathbb{R}^d$, with $d\geq 1$. The fractional Laplacian $(-\Delta)^s$ is defined through the Fourier transform as follows:
\begin{equation}
    \mathcal{F}\left[(-\Delta)^s u\right](\xi) = |\xi|^{2s} \mathcal{F} u(\xi),
\end{equation}
where the Fourier transform is given by
\begin{equation}
   \mathcal{F} u(\xi) = \frac{1}{(2\pi)^{d/2}} \int_{\mathbb{R}^d} u(x) e^{-i \xi \cdot x} \,dx. 
\end{equation}

The fractional Laplacian $(-\Delta)^s$ is a self-adjoint operator
acting on the space $L^2\(\mathbb{R}^d\)$, characterized by a
quadratic form domain $H^s\(\mathbb{R}^d\)$ and an operator
domain $H^{2 s}\(\mathbb{R}^d\)$. The nonlocal operator
$(-\Delta)^s$ serves as the infinitesimal generators in the context of
Lévy stable diffusion processes, as outlined in
\cite{applebaum2004levy}. Fractional derivatives of the Laplacian have
applications in numerous equations in mathematical physics and related
disciplines, as proposed in \cite{Laskin2000,Laskin2002} in the case
of linear Schr\"odinger equations; see also
\cite{applebaum2004levy,david2004levy,guo2006some} and the associated
references. Recently, there has been a strong focus on studying
mathematical problems related to the fractional Laplacian purely from
a mathematical perspective. Regarding specifically fractional
nonlinear Schr\"odinger equations, important progress has been made in
e.g. \cite{bhattarai2016existence,cabre2014nonlinear,cho2013cauchy,CHHO14,Dinh2018,guo2008existence,guo2012existence,guo2010global}. %Unlike the Schr\"odinger equation, Schr\"odinger's estimates are insufficient for solving the fractional Schr\"odinger equation in $L^2$. The local smoothing effect and maximal function estimates are also necessary. Thus, the weighted Sobolev space will be utilized to consider the well-posedness for \eqref{eq:logNLS}.

%The classical logarithmic Nonlinear Schr\"odinger (NLS) equation,
%proposed by Bialynicki-Birula and Mycielski in 1976, serves as a
%foundational model in nonlinear wave mechanics. It finds crucial
%applications in quantum mechanics, quantum optics, nuclear physics,
%open quantum systems, and Bose-Einstein condensation, as in
%\cite{hefter1985application,Zlo10} and so on.

The problem \eqref{eq:logNLS} does not seem to have physical
motivations (so far), and was introduced in \cite{d2015fractional} as
a generalization of the 
case $s=1$, introduced in \cite{BiMy76}, and proposed in different physical
contexts since (see e.g.
\cite{hefter1985application,Zlo10}). Note also
that the logarithmic nonlinearity may be obtained as the limit of an
homogeneous nonlinearity  $\lambda|u|^{2\si}u$ when $\si$ goes to
zero, at least when ground states are considered (case $\lambda<0$;
see \cite{WangZhang2019} for $s=1$, \cite{AnYang2023} in the
fractional case).

In \cite{Ardila2017}, the author addresses the nonlinear fractional
logarithmic Schr\"odinger equation \eqref{eq:logNLS} with
$\lambda=-1$ and $d\ge 2$, employing a compactness method to establish a unique
global solution for the associated Cauchy problem within a suitable
functional framework, inspired by \cite{CaHa80} (for the logarithmic
nonlinearity) and \cite{CHHO14} (for the fractional Laplacian). In
\cite{ZHANG2018}, the author investigate the 
existence of a global weak solution to the problem \eqref{eq:logNLS}
in the case of $\lambda=-1$, when the space variable $x$ belongs to
some smooth bounded domain, by using a combination of potential wells
theory and the Galerkin method.  In this paper, we
complement the approach from \cite{Ardila2017,ZHANG2018} by adapting
the strategy employed in \cite{HO} in the case of the standard
Laplacian, $s=1$.
\smallbreak

Formally, \eqref{eq:logNLS} enjoys the the conservations of mass,
angular momentum, and energy:
\begin{align} 
  & M(u(t))=\|u(t)\|_{L^2(\R^d)}^2,\notag \\
  & J(u(t))=
\IM\int_{\R^d}\bar u(t,x)\nabla u(t,x)dx,\notag\\ 
  &E(u(t))=\frac{1}{2}\|(-\Delta)^{s/2} u(t)\|_{L^2(\R^d)}^2+
    \frac{\lambda}{2}\int_{\R^d}
|u(t,x)|^2\(\log|u(t,x)|^2-1\)dx.\label{eq:energy}
\end{align} 
The energy is well-defined in the subset of $H^{s}\(\mathbb{R}^{d}\)$,
\begin{equation*} 
W_1^s:= \Big \{ u\in H^s(\R^d) \, , \,  x\mapsto
|u(x)|^2\log |u(x)|^2\in L^1(\R^d)\Big\} \, . 
\end{equation*}
When $s=1$, Hayashi and Ozawa \cite{HO} revisit the Cauchy problem
for the logarithmic Schr\"odinger equation, constructing strong
solutions in both $H^1$ and $W_1=W_1^1$. This approach deliberately
avoids relying on 
compactness arguments, demonstrating the convergence of a sequence of
approximate solutions in a complete function space. The authors in
\cite{HO}   also address the existence in the $H^2$-energy space, as
discussed below. 
\smallbreak

The main contributions  of this paper can be
summarized as follows:\\
1. Construction of $H^s$ strong solutions, without relying on the
conservation of the energy.\\ 
2. Construction of solutions in the energy space $W_1^s$.\\
3. The higher $H^{2s}$ regularity is established, by assuming some
further spatial decay 
of the initial data.

In all cases, no sign assumption is made on $\lambda\in \R$.

\begin{theorem}\label{theo:Hs}
    Let $\lambda \in \mathbb{R} $ and $0<s<1$. For any $\varphi \in H^{s}\(\mathbb{R}^{d}\)$, there exists a unique solution  $C\(\mathbb{R}, H^{s}\(\mathbb{R}^{d}\)\)$ to \eqref{eq:logNLS} in the sense of 
\begin{equation}\label{eq:logNLS in H^-s}
    i \d_{t} u-(-\Delta)^s u=\lambda  \log \(|u|^{2}\)u \quad \text { in } H^{-s}(\Omega)
\end{equation}
for all bounded open sets $\Omega \subset \mathbb{R}^{d}$ and all $t
\in \mathbb{R}$, and with $u_{\mid t=0}=\varphi$. If in addition  we assume
$\varphi \in W_{1}^s$, this $H^{s}$-solution satisfies $u
\in\(C \cap L^{\infty}\)\(\mathbb{R}, W_{1}^s\)$ if
$\lambda<0$ and $u \in C\(\mathbb{R}, W_{1}^s\)$ if
$\lambda>0$. Moreover, the $W_{1}^s$-solution $u$ satisfies the
equation \eqref{eq:logNLS in H^-s} in the sense of $({W_{1}^s})^{*}$,
where $(W_{1}^s)^{*}$ is the dual space of $W_{1}^s$. Finally, if
$\varphi\in H^1(\R^d)$, then the solution $u\in C\(\mathbb{R},
H^{s}\(\mathbb{R}^{d}\)\)$ to \eqref{eq:logNLS} satisfies in addition
$u\in C\(\mathbb{R}, H^1\(\mathbb{R}^{d}\)\)$.
\end{theorem}

The next result addresses on the construction of strong solutions in
$W_{2}^s$, where 
\begin{equation*} 
W_2^{s}:= \Big \{ u\in H^{2s}(\R^d) \, , \,  x\mapsto
u(x)\log |u(x)|^2\in L^2(\R^d)\Big\} \, ,
\end{equation*}
and this space is the natural counterpart of the space $W_2$ of the
$H^2$-energy space 
introduced in \cite{HO} for the case $s=1$. Note that considering this
space is interesting especially when $s>1/2$, since we have seen in
Theorem~\ref{theo:Hs} that the $H^1$ regularity is propagated, and
$H^1(\R^d)\subset H^{2s}(\R^d)$ when $s\le 1/2$. 

In the fractional case, it seems delicate to adapt the
strategy introduced in \cite{HO}, as some algebraic structure is
lost. More precisely, the strategy in \cite{HO} starts by showing that
$\d_t u \in L^\infty_{\rm loc}(\R,L^2)$, to eventually conclude that
$\Delta u \in L^\infty_{\rm loc}(\R,L^2)$. At this level of
generality, this is the standard approach, as presented in
e.g. \cite{CazCourant}, but the logarithmic nonlinearity actually
requires some special care. The above line of reasoning needs, as an
intermediary step, to know that $u\log|u|^2\in L^\infty_{\rm
  loc}(\R,L^2)$, which is by no means obvious, due to the region
$\{|u|<1\}$ where the nonlinearity is morally sublinear. This
difficulty is overcome in \cite{HO} by a beautiful algebraic identity
(\cite[Lemma~3.3]{HO}), whose derivation involves an integration by
parts in the term
\begin{equation*}
  \RE\( \Delta u,u\log\(|u|+\eps\)\)_{L^2} =- \RE\( \overline u \nabla
  u,\frac{\nabla|u|}{|u|+\eps}\)_{L^2}+ \( |\nabla u|^2,  \log\(|u|+\eps\)\)_{L^2} .
\end{equation*}
In the present case, we would face
\begin{equation*}
  \RE\( (-\Delta)^s u,u\log\(|u|+\eps\)\)_{L^2},
\end{equation*}
and the integration by parts would require to control a fractional
derivative of $ u\log\(|u|+\eps\)$, at least in the case $s<1/2$ (for
$s>1/2$, one could consider the gradient again). 
\smallbreak

To overcome this issue, we adopt the approach considered in
\cite{CaGa18} for the case $s=1$, and rely on some finite momentum
assumption.
For $0<\alpha \leq 1$, we have
\[
\mathcal{F}\(H^{\alpha}\)=\left\{u \in L^{2}\(\mathbb{R}^{d}\), x \mapsto\langle x\rangle^{\alpha} u(x) \in L^{2}\(\mathbb{R}^{d}\)\right\},
\]
where $\langle x\rangle:=\sqrt{1+|x|^{2}}$, and this space is equipped
with the norm
\[
\|u\|_{\mathcal{F}\(H^{\alpha}\)}:=\left\|\langle x\rangle^{\alpha} u(x)\right\|_{L^{2}\(\mathbb{R}^{d}\)} .
\]

Denote, for $\alpha>0$,  $X_\alpha^{2s}:=H^{2s}\cap
\mathcal{F}(H^\alpha)$:
for any $\alpha>0$, $X_\alpha^{2s}\subset W_2^s$, as can be seen from
the estimate, valid for any $\delta \in (0,1)$,
\begin{equation}\label{ineq:log_roughly}
\left|u \log \(|u|^2\)\right| \lesssim|u|^{1-\delta}+|u|^{1+\delta}.
\end{equation}

\begin{theorem}\label{theo:X}
    Let $\lambda \in \mathbb{R} $, $0<s<1$. Consider
    $0<\alpha<2s$ with $\alpha\le 1$. For any $\varphi \in
    X_\alpha^{2s}= H^{2s}\cap
\mathcal{F}(H^\alpha) $, there exists a unique solution $u\in
C_w\cap L^\infty_{\rm loc}(\R,X_\alpha^{2s})$ to \eqref{eq:logNLS} in the sense of 
    \begin{equation}\label{eq:logNLS in L^2(omega)} 
i\d_t u -(-\Delta )^s u =\lambda \log\(|u|^2\)u\, \quad \text{in} \quad L^2(\Omega), 
\end{equation} 
for all bounded open sets $\Omega \subset \mathbb{R}^{d}$ and a.e. $t
\in \mathbb{R}$, with $u_{\mid t=0}=\varphi$. Moreover, when $\lambda<0, u \in C\(\mathbb{R}, X_\alpha^{2s}\)$ and \eqref{eq:logNLS in L^2(omega)} holds in $L^{2}\(\mathbb{R}^{d}\)$ and for all $t \in \mathbb{R}$.
\end{theorem}
The new difficulty in proving the above result, compared to the case
$s=1$, lies in the fact that the Lie bracket
$[(-\Delta)^s,\<x\>^\alpha]$ requires some extra care; see
Lemma~\ref{lem:commutator}.

We underline the fact that we do not know whether the $H^\si$
regularity is propagated by the flow of \eqref{eq:logNLS}, where $\si
= \max (1,2s)$, like in the case of the regular Laplacian $s=1$.

\subsubsection*{Notations}
\begin{itemize}
\item $\int f$ is employed in place of $\int_{\mathbb R^d}f(x) dx$.
 \item The inner product in $L^2$ is denoted by
\[(f,g)_{L^2}=\int_{\mathbb R^d} f(x)\overline{g(x)}
 dx=\int f\bar{g}.\]
\item  Let $C(I,X)$ (resp. $C_w(I,X)$) be the space strongly
(resp. weakly) continuous functions from interval $I(\subseteq{\mathbb
  R})$ to $X$.\\ 
\item Abbreviated notation: for $T>0$, we write
\[C_T(X)=C([-T,T],X), \quad L_T^\infty(X)=L^\infty((-T,T),X),.
  \]
\item $A\lesssim B$ represents the inequality $A\leq CB$ with some
constant $C>0$.\
\end{itemize}

\subsubsection*{Content}

The rest of the paper is organized as follows. In
Section~\ref{sec:lem}, we collect lemmas which are of constant use in
this paper. Section~\ref{sec:Hs} is dedicated to the study of the Cauchy
problem for \eqref{eq:logNLS} in both $H^s$  and $W^1_s$, proving
Theorem~\ref{theo:Hs}. In Section~\ref{sec:X}, we  
consider higher regularity and prove Theorem~\ref{theo:X}.

\section{Useful lemmas}
\label{sec:lem}

The following lemma is a generalization of the inequality proven
initially by Cazenave and Haraux \cite{CaHa80} in the case $\eps=\mu=0$:

\begin{lemma}[\cite{HO}, Lemma A.1]\label{log}
    For all $u, v \in \mathbb{C}$ and $\eps, \mu\ge 0$, we have
    \begin{equation}\label{ineq:log}
|\IM(u \log (|u|+\eps)-v \log (|v|+\mu))(\bar{u}-\bar{v})| \leq|u-v|^2+|\eps-\mu||u-v|.
\end{equation}
\end{lemma}

We will also use several times the fractional Leibniz rule. We state a
simplified version of a result from \cite{Li19}, by using the
fact that BMO contains $L^\infty$, and considering only the $L^2$ setting.

\begin{lemma}[\cite{Li19}, Corollary~1.4]\label{lem:li19}
  For $\si>0$, let $A^\si$ be a differential operator such that its
  symbol $\widehat{A^\si}(\xi)$ is homogeneous of degree $\si$ and
  $\widehat{A^\si}(\xi)\in C^\infty(\Sph^{d-1})$.
  \begin{itemize}
  \item If $0<\si<1$,
    \begin{align*}
      &\|A^\si(fg)-gA^\si f\|_{L^2}\lesssim \|f\|_{L^2}\|(-\Delta)^{\si/2}
        g\|_{L^\infty}.
  %    &\|A^\si(fg)-fA^\si g-gA^\si f \|_{L^2}\lesssim
   %     \|f\|_{L^\infty}\|(-\Delta)^{\si/2}  g\|_{L^2}.
    \end{align*}
    \item If $1\le \si<2$, 
     \begin{align*}
      &\|A^\si(fg)-gA^\si f-\nabla g\cdot
        A^{\si,\nabla}f\|_{L^2}\lesssim \|f\|_{L^2}\|(-\Delta)^{\si/2}
        g\|_{L^\infty}.
   %   &\|A^\si(fg)-fA^\si g-gA^\si f -\nabla g\cdot
    %    A^{\si,\nabla}f \|_{L^2}\lesssim \|f\|_{L^\infty}\|(-\Delta)^{\si/2}
    %    g\|_{L^2}.,
     \end{align*}
     where $\widehat A^{\si,\nabla}(\xi) = -i\nabla_\xi
     \(\widehat{A^\si}(\xi)\)$. 
 \end{itemize}
  \end{lemma}

We recall that the characterization of the $H^s$ norm, for $0<s<1$,
can be expressed as follows (see e.g. \cite{Hitch2012}):
\begin{align*}
    \|f\|_{{H}^s}^2&=\|f\|_{L^2}^2+
                     \iint_{\R^d\times\R^d}\frac{|f(x)-f(y)|^2}{|x-y|^{d+2s}}dydx\\
  & =\|f\|_{L^2}^2+
                     \iint_{\R^d\times\R^d}\frac{|f(x+y)-f(x)|^2}{|y|^{d+2s}}dydx.
\end{align*}
We also have, for $0<s<1$ and $f\in \Sch(\R^d)$ (see e.g. \cite{Hitch2012}),
\begin{equation}\label{eq:lapl-frac-int}
  (-\Delta)^s f(x) =c(d,s)
  \iint_{\R^d\times\R^d}\frac{f(x+y)+f(x-y)-2f(x)}{|y|^{d+2s}}dydx, 
\end{equation}
for some constant $c(d,s)$ whose exact value is irrelevant here.

The following lemma will be crucial in the proof of
Theorem~\ref{theo:X}. 
\begin{lemma}\label{lem:commutator}
 Let $0<s<1$.   If $0<\alpha<2s$ and $\alpha\leq 1$, then the
 commutator $[(-\Delta)^s,\langle x\rangle ^\alpha]$ is continuous from $H^{s}(\mathbb R^d)$ to $L^2(\mathbb R^d)$.
\end{lemma}

\begin{proof}
 The proof relies on the fractional Leibniz rule stated in
 Lemma~\ref{lem:li19}, with $A^\si=(-\Delta)^s$, hence $\si=2s$. Fix
 $f\in C^\infty_c(\R^d)$, and let $g(x)=\<x\>^\alpha$.\\

 We first show that under the assumptions of the lemma, $(-\Delta)^sg\in
 L^\infty(\R^d)$, by using the characterization \eqref{eq:lapl-frac-int}.
 In the region $\{|y|\geq 1\}$, we write, since $0<\alpha\le 1$,
 \begin{equation*}
   |\<x\pm y\>^\alpha-\<x\>^\alpha|\le |\<x\pm
   y\>-\<x\>|^\alpha\lesssim |y|^\alpha, 
 \end{equation*}
 hence
 \[\left|\int_{|y|\geq1}
     \frac{\<x+y\>^\alpha+\<x-y\>^\alpha-2\<x\>^\alpha}{|y|^{d+2s}}
     dy\right|\lesssim \int_{|y|\geq1}
     \frac{|y|^\alpha}{|y|^{d+2s}} dy<\infty,
 \]
provided that $\alpha<2s$. 
In the ball $\{|y|<1\}$,  Taylor's formula yields
\begin{equation*}
     \<x+y\>^\alpha+\<x-y\>^\alpha-2\<x\>^\alpha = \<\nabla^2g(x)y,y\>
     + \O\( |y|^3 \),
\end{equation*}
where the remainder is uniform in $x\in\R^d$, as the third order derivatives of $g$ are bounded. Also, the
Hessian of $g$ is bounded since $|\nabla^2g(x)|\lesssim
\<x\>^{\alpha-2}$, and
\[\left|\int_{|y|\leq1}
     \frac{\<x+y\>^\alpha+\<x-y\>^\alpha-2\<x\>^\alpha}{|y|^{d+2s}}
     dy\right|\lesssim \int_{|y|\leq1}
     \frac{|y|^2}{|y|^{d+2s}} dy<\infty,
 \]
since $s<1$. 
\smallbreak

\noindent {\bf First case: $0<s<1/2$.} In view of the first case in
Lemma~\ref{lem:li19},
\begin{equation*}
  \|[(-\Delta)^s,\langle x\rangle ^\alpha]f\|_{L^2}\lesssim
  \|f\|_{L^2}\|(-\Delta)^sg\|_{L^\infty}\lesssim \|f\|_{L^2},
\end{equation*}
and $[(-\Delta)^s,\langle x\rangle ^\alpha]$ is continuous from
$L^2(\R^d)$ to $L^2(\R^d)$.

\noindent {\bf Second case: $1/2\le s<1$.}  In view of the second case in
Lemma~\ref{lem:li19},
\begin{equation*}
  \|[(-\Delta)^s,\langle x\rangle ^\alpha]f\|_{L^2}\lesssim
 \|\nabla g\cdot A^{2s,\nabla}f\|_{L^2} + \|f\|_{L^2}\|(-\Delta)^sg\|_{L^\infty}.
\end{equation*}
In view of the definition of $A^{2s,\nabla}$, with $A^{2s}=(-\Delta)^s$, 
\begin{equation*}
  \|A^{2s,\nabla}f\|_{L^2}\lesssim \|f\|_{\dot H^{2s-1}}\lesssim \|f\|_{H^s},
\end{equation*}
since $0<s<1$. Recalling that since $\alpha\le 1$, $\nabla g\in
L^\infty$, the lemma is proved.
\end{proof}

\section{The Cauchy problem in $H^s$ and the energy space}
\label{sec:Hs}

In this section, we prove Theorem~\ref{theo:Hs},
by resuming the strategy of \cite{HO}, which requires very few
adaptations to treat this fractional case (essentially, the fractional
Leibniz rule).

\subsection{Approximate problems}\label{snbsec:approx in H^s}

For $\eps>0$, we consider the approximate equation
\begin{equation}\label{approxi}
i \d_{t} u_{\eps}-(-\Delta)^s u_{\eps}=2 \lambda u_{\eps} \log \(\left|u_{\eps}\right|+\eps\), \quad
u_{\eps}(0, x)=\varphi(x).
\end{equation}
We set
\[
g(u)=2 u \log |u|, \quad g_{\eps}(u)=2 u \log (|u|+\eps).
\]
For $\si \geq 0$ we have
\[
\int_{0}^{\si} g_{\eps}(\tau) d \tau=\frac{1}{2} \si^{2} \log \((\si+\eps)^{2}\)-\frac{1}{2} \int_{0}^{\si} \frac{2 \tau^{2}}{\tau+\eps} d \tau.
\]
We define $G_{\eps}(u)$ by
\[
G_{\eps}(u)=\frac{1}{2} \int|u|^{2} \log \((|u|+\eps)^{2}\)-\frac{1}{2} \int \mu_{\eps}(|u|) ,\quad \text { for } u \in H^{s}\(\mathbb{R}^{d}\) \text {, }
\]
where
\[
\mu_{\eps}(\si):=\int_{0}^{\si} \frac{2 \tau^{2}}{\tau+\eps} d \tau \quad \text { for } \si \geq 0.
\]
We define $E_{\eps}(u)$ by
\begin{equation}\label{eq:E_varepsilon}
    \begin{aligned}
E_{\eps}(u) & =\frac{1}{2} \int|(-\Delta)^{s/2} u|^{2}+\lambda G_{\eps}(u) \\
& =\frac{1}{2} \int|(-\Delta)^{s/2} u|^{2}+\frac{\lambda}{2} \int|u|^{2} \log \((|u|+\eps)^{2}\)-\frac{\lambda}{2} \int \mu_{\eps}(|u|)
\end{aligned}
\end{equation}

\begin{lemma}\label{lem:exist-approx}
Let $\varphi\in H^s(\R^d)$ and $\eps>0$. Then \eqref{approxi}
possesses a unique solution
\begin{equation*}
  u_{\eps} \in C\(\mathbb{R}, H^{s}\(\mathbb{R}^{d}\)\) \cap C^{1}\(\mathbb{R}, H^{-s}\(\mathbb{R}^{d}\)\).
\end{equation*}
 Moreover, the mass and energy are conserved: for all $t\in \R$,
\begin{equation*}
    \left\|u_{\eps}(t)\right\|_{L^{2}}^{2}=\|\varphi\|_{L^{2}}^{2},\quad 
    E_{\eps}\(u_{\eps}(t)\)=E_{\eps}(\varphi).
\end{equation*}
\end{lemma}
\begin{proof}
  Unlike in the case of the regular Laplacian, $s=1$, it seems delicate
  to invoke Strichartz estimates independently of the space dimension
  $d$ in order to solve \eqref{approxi} in $H^s$, since a loss of
  regularity is present when $0<s<1$, see \cite{COX11}, and
  \cite{GuoWang2014,HongSire2015}.  We rather adopt the approach of
  \cite{CHHO14}, which in turn resumes the arguments from
  \cite{CazCourant}. A key step is to check that, for a
  given $T>0$,
  \eqref{approxi} has a most one (weak) solution $u_\eps\in
  L^\infty_TH^s\cap W^{1,\infty}_TH^{-s}$. 
By interpolation, such a
  solution belongs to $C_TL^2$, and if $u_\eps,v_\eps$ are two such
  solutions, $u_\eps-v_\eps$ solves
  \begin{equation*}
    \(i\d_t-(-\Delta)^s\)(u_\eps-v_\eps)=
    \lambda\(u_\eps\log\(|u_\eps|+\eps\) - v_\eps\log\(|v_\eps|+\eps\) \).
  \end{equation*}
We then proceed with the usual argument for $L^2$ estimates in
Schr\"odinger equations: multiply by $\bar u_\eps-\bar v_\eps$,
integrate over $\R^d$, and take the imaginary part. The term
involving the fractional Laplacian vanishes by self-adjointness, and
the nonlinear term is controlled thanks to Lemma~\ref{log} (with
$\mu=\eps$), so we get
\begin{equation*}
  \frac{d}{dt}\|u_\eps-v_\eps\|_{L^2}^2\lesssim \|u_\eps-v_\eps\|_{L^2}^2,
\end{equation*}
hence $u_\eps\equiv v_\eps$ by Gronwall Lemma, since
$\|u_\eps(t)-v_\eps(t)\|_{L^2}$ is continuous at $t=0$. The existence
of such a weak
  solution is given by \cite[Theorem~3.3.5]{CazCourant}, which is
  readily adapted to the case of the fractional Laplacian, and since
  we note that for fixed $\eps>0$, there exists a function $C^\eps(\cdot)$
  such that if $\|u\|_{H^s},\|v\|_{H^s}\le M$, then
  \begin{equation*}
 \|g_\eps(u)-g_\eps(v)\|_{H^{-s}} \le   \|g_\eps(u)-g_\eps(v)\|_{L^2}
 \le C^\eps(M) \|u-v\|_{H^s}.  
  \end{equation*}
With the above
uniqueness property, we can resume the proof of
\cite[Theorem~3.3.9]{CazCourant} and \cite[Theorem~3.4.1]{CazCourant}
for the globalization, since, as we have, for every $\delta>0$,
\begin{equation*}
  \left| |u|^2\log\(|u|+\eps\) -\mu_\eps(|u|)\right|\le
  C_{\eps,\delta} |u|^{2+\delta} + |u|,
\end{equation*}
Gagliardo-Nirenberg and Young inequalities yield
\begin{equation*}
 |\lambda G_\eps(u)| \le \frac{1}{4}\|u\|_{\dot H^s}^2 + C(\|u\|_{L^2}),
\end{equation*}
so we obtain the lemma. 
\end{proof}

\subsection{Construction of weak $H^s$ solution}
\label{sec:construct1}
We initially establish a uniform estimate for approximate solutions within the $H^s$ space.

\begin{lemma}\label{lem:prior estimate}
    Let $0<\alpha\leq 1$ and $\varphi\in H^s$. For all $t\in \mathbb R$ we have
    \begin{equation}\label{ineq:prior estimate}
\left\|(-\Delta)^{s/2} u_{\eps}(t)\right\|_{L^2}^2 \leq e^{4|\lambda||t|}\|(-\Delta)^{s/2} \varphi\|_{L^2}^2 .
\end{equation}
\end{lemma}
\begin{proof}
    We resume the energy estimate from \cite{carles2023low}: in view
    of the conservation of the $L^2$-norm,
    \begin{align*}
& \frac{d}{d t}\left\|u_{\eps}(t)\right\|_{H^{s}(\mathbb R^d)}^{2} \\
= & 2 \RE \iint_{\mathbb R^d \times \mathbb R^d} \overline{\(u_{\eps}(t, x+y)-u_{\eps}(t, x)\)} \d_{t}\(u_{\eps}(t, x+y)-u_{\eps}(t, x)\) \frac{d x d y}{|y|^{d+2 \alpha}} \\
= & -2\IM \iint_{\mathbb R^d \times \mathbb R^d} \overline{\(u_{\eps}(t, x+y)-u_{\eps}(t, x)\)} (-\Delta)^s\(u_{\eps}(t, x+y)-u_{\eps}(t, x)\) \frac{d x d y}{|y|^{d+2 \alpha}} \\
& -4 \lambda \IM \iint_{\mathbb R^d \times \mathbb R^d} \overline{\(u_{\eps}(t, x+y)-u_{\eps}(t, x)\)}\(g_{\eps}\(u_{\eps}(t, x+y)\)-g_{\eps}\(u_{\eps}(t, x)\)\) \frac{d x d y}{|y|^{d+2 \alpha}}.
\end{align*}
Here, the first term on the right-hand side of the equation vanishes,
since
$(-\Delta)^s$ is self-adjoint, and so the imaginary part of the
integral in $x$ is zero.
By applying Lemma~\ref{log} with $\mu=\eps$, we obtain 
\begin{align*}
& \frac{d}{d t}\left\|u_{\eps}(t)\right\|_{H^{s}(\mathbb R^d)}^{2} \\
\leq & 4|\lambda| \iint_{\mathbb R^d \times \mathbb R^d}\left|\IM\left[\overline{\(u_{\eps}(t, x+y)-u_{\eps}(t, x)\)}\(g_{\eps}\(u_{\eps}(t, x+y)\)-g_{\eps}\(u_{\eps}(t, x)\)\)\right]\right| \frac{d x d y}{|y|^{d+2 \alpha}} \\
\leq & 4|\lambda| \iint_{\mathbb R^d \times \mathbb R^d}\left|u_{\eps}(t, x+y)-u_{\eps}(t, x)\right|^{2} \frac{d x d y}{|y|^{d+2 \alpha}} \leq 4|\lambda|\left\|u_{\eps}(t)\right\|_{H^{s}(\mathbb R^d)}^{2} .
\end{align*}
Gronwall Lemma then yields
\[
\left\|u_{\eps}(t)\right\|_{H^{s}(\mathbb R^d)}^{2} \leq e^{4|\lambda t|}\|\varphi\|_{H^{s}(\mathbb R^d)}^{2} \quad \text { for all } t \in \mathbb{R} ,
\]
hence the lemma. 
\end{proof}

It follows from Lemma~\ref{lem:exist-approx} and \eqref{ineq:prior estimate} that for any $T>0$ we have
\begin{equation}\label{MT}
M_T:=\sup
_{0<\eps<1}\left\|u_{\eps}\right\|_{L^\infty_T\(H^s\)}
\leq C\(T,\|\varphi\|_{H^s}\) . 
\end{equation}
Next we prove that $\left\{u_{\eps}\right\}_{0<\eps<1}$ forms a Cauchy sequence in $C_T\(L_{\mathrm{loc}}^2\(\mathbb{R}^d\)\)$ as $\eps \downarrow 0$ for any $T>0$.
Take a function $\zeta \in C_c^{\infty}\(\mathbb{R}^d\)$ satisfying
\[
\zeta(x)=\left\{\begin{array}{ll}
1 & \text { if }|x| \leq 1, \\
0 & \text { if }|x| \geq 2, \\
\end{array} \quad 0 \leq \zeta(x) \leq 1 \quad \text { for all } \quad x \in \mathbb{R}^d.\right.
\]
For $R>0$ we set $\zeta_R:=\zeta(x / R)$. For $\eps, \mu \in(0,1)$, utilizing \eqref{approxi}, \eqref{log}, and \eqref{MT}, a direct computation indicates that
\[
\begin{aligned}
\frac{d}{d t}\left\|\zeta_R\(u_{\eps}-u_\mu\)\right\|_{L^2}^2= & 2 \IM\(i \zeta_R^2 \d_t\(u_{\eps}-u_\mu\), u_{\eps}-u_\mu\) \\
= & 2 \IM\(\zeta_R^2(-\Delta)^s\(u_{\eps}-u_\mu\), u_{\eps}-u_\mu\) \\
& +4 \lambda \IM\(\zeta_R^2\(u_{\eps}
    \log \(\left|u_{\eps}\right|+\eps\)-u_\mu
    \log \(\left|u_\mu\right|+\mu\)\),
  u_{\eps}-u_\mu\) .
\end{aligned}
\]
The first term on the right hand side is estimated thanks to  the
fractional Leibniz rule recalled in (the first
case of) Lemma~\ref{lem:li19}, since
\begin{equation*}
  \IM \( (-\Delta)^{s/2}\(u_{\eps}-u_\mu\), 
 \zeta_R^2(-\Delta)^{s/2}( u_{\eps}-u_\mu)\)=0,
\end{equation*}
by:
\begin{align*}
 & \left|\IM
  \((-\Delta)^{s/2}\(u_{\eps}-u_\mu\), 
 (-\Delta)^{s/2}\(\zeta_R^2( u_{\eps}-u_\mu)\)\)   \right|\\
  &\quad \lesssim
    \|u_{\eps}-u_\mu\|_{\dot{H^s}}\|u_{\eps}-u_\mu\|_{L^2}
    \|(-\Delta)^{s/2}\(\zeta_R^2\)\|_{L^\infty} . 
\end{align*}
The estimate $\|(-\Delta)^{s/2}\zeta_R^2\|_{L^\infty}\lesssim
1/R^s$ follows by homogeneity (using e.g. Fourier transform), and thus
\begin{align*}
  \frac{d}{d
    t}\left\|\zeta_R\(u_{\eps}-u_\mu\)\right\|_{L^2}^2&
     \le\frac{C}{R^s}\left\|u_{\eps}-u_\mu\right\|_{\dot{H^s}}
    \left\|u_{\eps}-u_\mu\right\|_{L^2}\\
  &\quad
    +4|\lambda|\(\left\|\zeta_R\(u_{\eps}-u_\mu\)\right\|_{L^2}^2
    +|\eps-\mu|
    \left\|\zeta_R^2\(u_{\eps}-u_\mu\)\right\|_{L^1}\).
\end{align*}
Gronwall Lemma implies
\begin{equation}\label{ineq:zetaL^2}
\left\|\zeta_R\(u_{\eps}-u_\mu\)(t)\right\|_{L^2}^2 \leq e^{4|\lambda| T} \(\frac{C\(M_T\)}{R^s}+|\eps-\mu|\left|B_{2 R}\right|^{1 / 2}\|\varphi\|_{L^2}\),
\end{equation}
for all $t\in [-T,T]$, where we have used
\[
\left\|\zeta_R^2\(u_{\eps}-u_\mu\)\right\|_{L^1} \leq\left\|u_{\eps}-u_\mu\right\|_{L^2\(B_{2 R}\)} \leq 2\left|B_{2 R}\right|^{1 / 2}\|\varphi\|_{L^2}.
\]
We now fix $R_0>0$ and take $R \in\(R_0, \infty\)$ as a parameter. It follows from \eqref{ineq:zetaL^2} that
\[
\left\|u_{\eps}-u_\mu\right\|_{C_T\(L^2\(B_{R_0}\)\)}^2 \leq\left\|\zeta_R\(u_{\eps}-u_\mu\)\right\|_{C_T\(L^2\)}^2 \leq C\(T,\|\varphi\|_{H^s}\)\(\frac{1}{R^s}+|\eps-\mu|\left|B_{2 R}\right|^{1 / 2}\),
\]
which yields 
\[
\limsup _{\eps, \mu \downarrow 0}\left\|u_{\eps}-u_\mu\right\|_{C_T\(L^2\(B_{R_0}\)\)}^2 \leq \frac{C\(T,\|\varphi\|_{H^s}\)}{R^s} \underset{R \to \infty}{\longrightarrow} 0.
\]
As $R_0>0$ is arbitrary, we conclude that the sequence
$\left\{u_{\eps}\right\}_{0<\eps<1}$ constitutes a
Cauchy sequence in
$C_T\(L_{\mathrm{loc}}^2\(\mathbb{R}^d\)\)$. When
combining this with Lemma~\ref{lem:exist-approx},  this entails that there exists a function $u \in L^{\infty}\(\mathbb{R}, L^2\(\mathbb{R}^d\)\)$ such that
\begin{equation}\label{converge in L^2_loc}
u_{\eps} \to u \quad \text { in } C_T\(L_{\text {loc }}^2\(\mathbb{R}^d\)\) \quad \text { as } \eps \downarrow 0,
\end{equation}
for all $T>0$.
\begin{lemma}\label{converg in H^s}
    $u \in L_{\mathrm{loc}}^{\infty}\(\mathbb{R}, H^s\(\mathbb{R}^d\)\)$ and
    \begin{equation}\label{weak converg in H^s}
u_{\eps}(t) \rightharpoonup u(t) \quad \text { in } H^s\(\mathbb{R}^d\) \quad \text { for all } t \in \mathbb{R} .
\end{equation}
\end{lemma}
\begin{proof}
    First it follows from \eqref{converge in L^2_loc} that
    \begin{equation}\label{weak converg in L^2}
u_{\eps}(t) \rightharpoonup u(t) \quad \text { in } L^{2}\(\mathbb{R}^{d}\) \quad \text { for all } t \in \mathbb{R} \text {. }
\end{equation}

To prove $u \in L_{\mathrm{loc}}^{\infty}\(\mathbb{R},
  H^s\(\mathbb{R}^d\)\)$, we use the characterization of
$H^s$ functions by duality. For any $\psi \in
C_c^\infty\(\mathbb{R}^d\)$ and $t\in [-T,T]$ we obtain from
\eqref{MT} that  
\begin{align*}
\left|\int u_{\eps}(t) (-\Delta)^{s/2} \psi\right|=\left|\int
  (-\Delta)^{s/2} u_{\eps}(t) \psi\right| \le
  \|u_\eps(t)\|_{\dot H^s}\|\psi\|_{L^2}
 \le M_T\|\psi\|_{L^2}.
\end{align*}
Then it follows from \eqref{weak converg in L^2} that
\[
\left|\int u(t) (-\Delta)^{s/2}\psi\right| \leq
M_T\|\psi\|_{L^2} \quad \text { for all } t \in[-T, T]
. 
\]
We infer that for all $t\in [-T,T]$,
\begin{equation*}
  u(t)\in H^s(\R^d)\quad\text{and}\quad
  \|(-\Delta)^{s/2}u(t)\|_{L^2}\le M_T,
\end{equation*}
hence $u\in L^\infty_{\rm loc}(\R,H^s(\R^d))$. 
Also, in view of \eqref{weak converg in L^2},
$$\int (-\Delta)^{s/2}u_{\eps}(t) \psi \to \int (-\Delta)^{s/2} u(t) \psi ,$$
for any $\psi\in C_c^\infty({\mathbb R}^d)$ and $t\in \mathbb R$.
Using \eqref{MT} and a density argument, we deduce that
\begin{equation*}%\label{weak converg in H^s}
(-\Delta)^{s/2} u_{\eps}(t) \rightharpoonup (-\Delta)^{s/2}
u(t) \quad \text { in } L^2\(\mathbb{R}^d\),\quad\text{for
  all }t\in \R,
\end{equation*}
hence the lemma.
\end{proof}
Next, we prove that the convergence of the nonlinear term.
\begin{lemma}\label{converg g in L^2_loc}
  For all $t\in \mathbb R$ we have
  \begin{equation*}
    g_{\eps}\(u_{\eps}(t)\) \to g(u(t))
    \text { in } L_{\mathrm{loc}}^{2}\(\mathbb{R}^{d}\) \quad
    \text { as } \eps \downarrow 0. 
  \end{equation*}
\end{lemma}
\begin{proof}
    We show that for any $\Omega \subset \subset \mathbb{R}^{d}$ and $t \in \mathbb{R}$,  
    $$
u_{\eps}(t) \log \(\left|u_{\eps}(t)\right|+\eps\) \to u(t) \log |u(t)| \quad \text { in } L^{2}(\Omega) \quad \text { as } \eps \downarrow 0 \text {. }
$$
In view of \cite[Lemma~A.2]{HO}, we know that for $\alpha\in (0,1)$, there exists $C(\alpha)>0$ such that for all $u,v\in \mathbb C$, $\eps\in (0,1)$
\begin{align*}
|v \log (|v|+\eps)-u \log | u|| \leq & \eps+|u-v|+C(\alpha) \times \\
& \(1+|u|^{1-\alpha} \log ^{+}|u|+|v|^{1-\alpha} \log ^{+}|v|\)|u-v|^{\alpha},
\end{align*}
where $\log ^{+} x:=\max (\log x, 0)$.
Hence, for any $\delta>0$ small, there exists $C(\delta)>0$ such that 
$$
\begin{aligned}
|u_\eps \log (|u_\eps|+\eps)-u \log | u|| \leq & \eps+|u_\eps-u|+C(\alpha) \times \\
& \(1+|u_\eps|^{\frac{1}{2}+\delta}+|u|^{\frac{1}{2}+\delta} \)|u_\eps-u|^{1/2}.
\end{aligned}
$$
Fixing $\delta>0$ sufficiently small so that $H^s(\R^d)\subset
L^{2+4\delta}(\R^d)$, we have
\begin{align*}
\left\|\left|u_{\eps}\right|^{\frac{1}{2}+\delta}\left|u_{\eps}-u\right|^{1
  / 2}\right\|_{L^{2}(\Omega)}^{2} &
    =\int_{\Omega}\left|u_{\eps}\right|^{1+2 \delta}
   \left|u_{\eps}-u\right| \le
   \left\|u_{\eps}\right\|_{L^{2+4 \delta}}^{1 +2\delta}
  \left\|u_{\eps}-u\right\|_{L^{2}(\Omega)} \\
& \lesssim \left\|u_{\eps}\right\|_{H^{s}}^{1+2\delta}\left\|u_{\eps}-u\right\|_{L^{2}(\Omega)}.
\end{align*}
Therefore, the results follows from \eqref{MT} and \eqref{converge in L^2_loc}.
\end{proof}

From \eqref{approxi} it follows that for every $\varphi\in
C_c^\infty({\mathbb R}^d)$ and every $\phi\in C_c^1(\mathbb R)$, 
\begin{align*}
\int_{\mathbb{R}}\(i u_{\eps}, \psi\)_{L^{2}} \phi^{\prime}(t) d t & =-\int_{\mathbb{R}}\left\langle i \d_{t} u_{\eps}, \psi\right\rangle_{H^{-s}, H^{s}} \phi(t) d t \\
& =-\int_{\mathbb{R}}\left\langle(-\Delta)^s u_{\eps}+2 \lambda u_{\eps} \log \(\left|u_{\eps}\right|+\eps\), \psi\right\rangle_{H^{-s}, H^{s}} \phi(t) d t \\
& =-\int_{\mathbb{R}}\left\{\((-\Delta)^{s/2} u_{\eps}, (-\Delta)^{s/2} \psi\)_{L^{2}}+\(\lambda g_{\eps}\(u_{\eps}\), \psi\)_{L^{2}}\right\} \phi(t) d t.
\end{align*}
From \eqref{weak converg in H^s}, $u_\eps(t)\rightharpoonup
u(t)$ in $H^s({\mathbb R}^d)$. In view of Lemma~\ref{converg g in
  L^2_loc}, taking the limit $\eps \downarrow 0$ yields
$$
\int_{\mathbb{R}}(i u, \psi)_{L^{2}} \phi^{\prime}(t) d t=-\int_{\mathbb{R}}\left\{((-\Delta)^{s/2} u, (-\Delta)^{s/2} \psi)_{L^{2}}+(\lambda g(u), \psi)_{L^{2}}\right\} \phi(t) d t.
$$
It can be easily verified for any $\Omega \subset \subset \mathbb{R}^{d}$,
$$
u \in L_{\mathrm{loc}}^{\infty}\(\mathbb{R}, H^{s}\(\mathbb{R}^{d}\)\) \cap W_{\mathrm{loc}}^{s, \infty}\(\mathbb{R}, H^{-s}(\Omega)\)
$$
%(Note that $\|u\|_{L^\infty}\leq \|u\|_{W^{s,p}}$.) 
and 
\begin{equation}\label{eq:H^-s}
    i \d_{t} u-(-\Delta)^s u=\lambda g(u) \quad \text { in } H^{-s}(\Omega) ,
\end{equation}
for almost all $t\in \R$.

\subsection{Uniqueness and regularity}
Following \cite[Lemme~2.2.1]{CaHa80}, we have:
\begin{lemma}
    Assume that, for some $T>0$,  $u, v \in
    L_{T}^{\infty}\(H^{s}\(\mathbb{R}^{d}\)\)$
    solve \eqref{eq:logNLS} in the distribution sense. Then $u=v$. 
\end{lemma}
\begin{proof}
    We set
    $$
M:=\max \left\{\|u\|_{L_{T}^{\infty}\(H^{s}\)},\|v\|_{L_{T}^{\infty}\(H^{s}\)}\right\} .
$$
As mentioned above, $u, v$ satisfy the equation in the sense of
\eqref{eq:H^-s}. Resuming the cut-off function $\zeta_{R}$, and the
computations from Section~\ref{sec:construct1} (with $u_\eps$ replaced
by $u$ and $u_\mu$ replaced by $v$), Gronwall Lemma yields, like for
\eqref{ineq:zetaL^2} (with now $\eps=\mu=0$),
$$
\left\|\zeta_{R}(u-v)(t)\right\|_{L^{2}}^{2} \leq e^{4|\lambda| T} \( \|\zeta_R(u(0)-v(0))\|^2_{L^2}+\frac{C(M)}{R^s}T \)\quad \text { for all } t \in[-T, T].
$$
By Fatou's Lemma,
$$
\|(u-v)(t)\|_{L^2}^2 \leq \liminf _{R \to \infty}\left\|\zeta_{R}(u-v)(t)\right\|_{L^{2}}^{2} \leq 0,
$$
for all $t\in [-T,T]$. Therefore, $u=v$ on $[-T,T]$.
\end{proof}
Continuity in time and strong $L^2$ convergence are established like
in the proof of \cite[Lemma 2.10]{HO}. 
\begin{lemma}\label{converg in L^2}
$u \in C_{w}\(\mathbb{R}, H^{s}\(\mathbb{R}^{d}\)\) \cap C\(\mathbb{R}, L^{2}\(\mathbb{R}^{d}\)\)$ and
$$
u_{\eps}(t) \to u(t) \quad \text { in } L^{2}\(\mathbb{R}^{d}\).
$$

\end{lemma}
\begin{proof}
        First we note that $u \in C_{w}\(\mathbb{R},
          H^{s}\(\mathbb{R}^{d}\)\)$. Indeed this easily
        follows from Lemma \ref{converg in H^s} and $u \in
        C\(\mathbb{R}, L_{\text {loc
            }}^{2}\(\mathbb{R}^{d}\)\)$. Next, we obtain
        from Lemma~\ref{lem:exist-approx} and \eqref{weak converg in L^2} that
    $$
\|u(t)\|_{L^{2}}^{2} \leq \liminf _{\eps \to 0}\left\|u_{\eps}(t)\right\|_{L^{2}}^{2}=\|\varphi\|_{L^{2}}^{2} \quad \text { for all } t \in \mathbb{R} \text {. }
$$
Uniqueness of solutions yields that
\begin{equation}\label{eq:u_L^2}
   \|u(t)\|_{L^{2}}^{2}=\|\varphi\|_{L^{2}}^{2} \quad \text { for all } t \in \mathbb{R}
\end{equation}
As $u \in C_{w}\(\mathbb{R},
  L^{2}\(\mathbb{R}^{d}\)\)$, we deduce that $u \in
C\(\mathbb{R}, L^{2}\(\mathbb{R}^{d}\)\)$. Since no
mass is lost in the weak convergence \eqref{weak converg in L^2}, the
convergence is strong in $L^2$. 
\end{proof}

\begin{lemma}\label{u in C(H^s)}
    $u \in C\(\mathbb{R}, H^{s}\(\mathbb{R}^{d}\)\)$.
\end{lemma}
\begin{proof}
    We just need to show the continuity $t \mapsto u(t) \in H^{s}\(\mathbb{R}^{d}\)$ at $t=0$. It follows  from \eqref{ineq:prior estimate}, \eqref{weak converg in H^s}, and the weak lower semicontinuity of the norm that
\[
\|u(t)\|_{\dot{H^{s}}}^{2} \leq e^{4|\lambda||t|}\|\varphi\|_{\dot{H^{s}}}^{2} .
\]
Passing to the limit as $t\to 0$ we have
$$
\limsup _{t \to 0}\|u(t)\|_{\dot{H^{s}}}^{2} \leq\|\varphi\|_{\dot{H^{s}}}^{2}.
$$
On the other hand, it follows from the weak continuity $t \mapsto u(t) \in H^{s}\(\mathbb{R}^{d}\)$ at $t=0$ that
$$
\|\varphi\|_{\dot{H^{s}}}^{2} \leq \liminf _{t \to 0}\|u(t)\|_{\dot{H^{s}}}^{2}.
$$
So we obtain
$$
\lim _{t \to 0}\|u(t)\|_{\dot{H^{s}}}^{2}=\|\varphi\|_{\dot{H^{s}}}^{2}.
$$
Therefore, the weak convergence in \eqref{weak converg in H^s} is
actually strong. 
\end{proof}

\subsection{Construction of solutions in $W_{1}^s$}
We now assume that $\varphi \in W_{1}^s \subset H^{s}\(\mathbb{R}^{d}\)$. From the dominated convergence theorem we have
$$
E_{\eps}(\varphi) \to E(\varphi) \quad \text { as } \eps \downarrow 0,
$$
recalling that $E_{\eps}(\varphi)$ and $E(\varphi)$ are defined
in \eqref{eq:E_varepsilon} and \eqref{eq:energy}, respectively.
Let $\theta \in C_{c}^{1}(\mathbb{C},
\mathbb{R})$ satisfying 
\[
\theta(z)=\left\{\begin{array}{ll}
1 & \text { if }|z| \leq 1 / 4, \\
0 & \text { if }|z| \geq 1 / 2,
\end{array} \quad 0 \leq \theta(z) \leq 1 \quad \text { for } z \in \mathbb{C},\right.
\]
and set, for $\eps>0$,
\begin{align*}
F_{1 \eps}(u) & =\theta(u)|u|^{2} \log
                       \((|u|+\eps)^{2}\), \quad F_{2 \eps}(u)=(1-\theta(u))|u|^{2} \log \((|u|+\eps)^{2}\), \\
F_{1}(u) & =\theta(u)|u|^{2} \log \(|u|^{2}\),\quad  F_{2}(u)=(1-\theta(u))|u|^{2} \log \(|u|^{2}\) .
\end{align*}
In the subsequent discussion, we confine the range of $\eps$ to  $(0, 1/2)$. The energy expressed in equation \eqref{approxi} is denoted as
$$
E_{\eps}(u)=\frac{1}{2} \int|(-\Delta)^{s/2} u|^{2}+\frac{\lambda}{2} \int F_{1 \eps}(u)+\frac{\lambda}{2} \int F_{2 \eps}(u)-\frac{\lambda}{2} \int \mu_{\eps}(|u|) .
$$
Taking $\delta>0$ sufficiently small, 
\begin{equation}\label{ineq:F_2}
    \int\left|F_{2}(u)\right| \lesssim \int |u|^{2+\delta} \lesssim
    \(\|u\|_{L^2}^{1-\eta}\|u\|_{\dot H^s}^\eta\)^{2+\delta}, \quad
    \eta= \frac{d}{s}\( \frac{1}{2}-\frac{1}{2+\delta}\)\in (0,1).
\end{equation}
In particular,
\begin{equation*}
  \text{for }u \in H^{s}\(\R^{d}\),\quad u \in
    W_{1}^s \Longleftrightarrow \int\left|F_{1}(u)\right|<\infty .
\end{equation*}

\begin{lemma}\label{lem:mu&F converg}
  For all $t \in \mathbb{R}$ we have, as $\eps\to 0$,
  \begin{equation*}
      \int \mu_{\eps}\(\left|u_{\eps}(t)\right|\)
      \to \int|u(t)|^{2}, \quad \int F_{2
        \eps}\(u_{\eps}(t)\) \to \int
      F_{2}(u(t)). 
  \end{equation*}
\end{lemma}
  The proof of this lemma is found in\cite[Lemma 2.13]{HO},  and
  relies on the observation that for any $\delta \in(0,1)$ there
  exists $C(\delta)>0$ such that
  \begin{equation*}
    \left|F_{2 \eps}(z)-F_{2}(w)\right| \leq C(\delta)\(|z|^{1+\delta}+|w|^{1+\delta}\)|z-w| \quad \text { for all } z, w \in \mathbb{C}.
  \end{equation*}
\begin{proposition}
    Let $\lambda< 0$. Then, $u \in\(C \cap L^{\infty}\)\(\mathbb{R}, W_{1}^s\)$ and $E(u(t))= E(\varphi)$ for all $t \in \mathbb{R}$.
\end{proposition}
\begin{proof}
    For $\eps \in(0,1 / 2)$, we have $F_{1 \eps}(u)\leq
    0$, and we can rewrite the  first two terms $E_{\eps}(u)$ as
    $$
\frac{1}{2} \int|(-\Delta)^{s/2} u|^{2}+\frac{\lambda}{2} \int F_{1 \eps}(u)=\frac{1}{2} \int|(-\Delta)^{s/2} u|^{2}+\frac{|\lambda|}{2} \int |F_{1 \eps}(u)|.
$$
The weak lower semicontinuity of the norm, Fatou's
lemma (for the second term), and  Lemma~\ref{lem:mu&F converg} imply
\begin{align*}
\frac{1}{2} \int|(-\Delta)^{s/2} u(t)|^{2}+\frac{|\lambda|}{2} \int\left|F_{1}(u(t))\right|& \leq \liminf _{\eps \to 0}\(E_{\eps}\(u_{\eps}(t)\)-\frac{\lambda}{2} \int F_{2 \eps}\(u_{\eps}(t)\)+\frac{\lambda}{2} \int \mu_{\eps}\(u_{\eps}(t)\)\) \\
& \leq E(\varphi)-\frac{\lambda}{2} \int F_{2}(u(t))+\frac{\lambda}{2} \int|u(t)|^{2},
\end{align*}
for all $t \in \mathbb{R}$. It implies that
$$
u(t) \in W_{1}^s, \quad E(u(t)) \leq E(\varphi) \quad \text { for all } t \in \mathbb{R}.
$$
Invoking Lemma~\ref{converg in L^2}, we obtain that the conservation
of the energy 
\begin{equation}\label{eq:E}
    E(u(t))=E(\varphi) \quad \text { for all } t \in \mathbb{R}.
\end{equation}
From inequality \eqref{ineq:F_2} with $(2+\delta)\eta<2$, and the identity \eqref{eq:u_L^2} we
get
\begin{equation*}
  \int|(-\Delta)^{s/2} u(t)|^{2}+\int\left|F_{1}(u(t))\right| \leq C\(
  E(\varphi),\|\varphi\|_{L^{2}}\),
\end{equation*}
for all $t \in \mathbb{R}$. Therefore we deduce that
\begin{equation*}
  u \in L^{\infty}\(\mathbb{R}, H^{s}\(\mathbb{R}^{d}\)\) \quad \text { and } \quad t \mapsto \int|u(t)|^{2} \log \(|u(t)|^{2}\) \in L^{\infty}(\mathbb{R}),
\end{equation*}
and thus $u \in L^{\infty}\(\mathbb{R}, W_{1}^s\)$. Moreover,
from \eqref{eq:E} and Lemma \eqref{u in C(H^s)}, we know
that
\begin{equation*}
  t \mapsto \int|u(t)|^{2} \log \(|u(t)|^{2}\) \in C(\mathbb{R}) \Longleftrightarrow u \in C\(\mathbb{R}, W_{1}^s\),
\end{equation*}
which completes the proof.
\end{proof}
\begin{proposition}
    Let $\lambda>0$. Then, $u \in C\(\mathbb{R}, W_{1}^s\)$.
\end{proposition}
\begin{proof}
    \textbf{Step 1.} We show that $u \in
    L_{\mathrm{loc}}^{\infty}\(\mathbb{R}, W_{1}^s\)$. It follows from
    \eqref{eq:E_varepsilon} and \eqref{eq:E} that for  any $T>0$ and
    $t \in[-T, T]$, 
\begin{align*}
\frac{|\lambda|}{2} \int\left|F_{1 \eps}\(u_{\eps}(t)\)\right| & =-\frac{\lambda}{2} \int F_{1 \eps}\(u_{\eps}(t)\) \\
& =-E_{\eps}\(u_{\eps}(t)\)+\frac{1}{2} \int\left|(-\Delta)^{s/2}
     u_{\eps}(t)\right|^{2} \\
&\quad+\frac{\lambda}{2}\(\int F_{2 \eps}\(u_{\eps}(t)\)-\int \mu_{\eps}\(\left|u_{\eps}(t)\right|\)\) .
\end{align*}
Fatou's Lemma and \eqref{MT} imply
\begin{equation*}
  \frac{|\lambda|}{2} \int\left|F_{1}(u(t))\right| \leq \liminf _{\eps \to 0} \frac{|\lambda|}{2} \int\left|F_{1 \eps}\(u_{\eps}(t)\)\right| \leq -E(\varphi)+C\(M_{T}\),
\end{equation*}
for all $t \in[-T, T]$. This entails
\begin{equation*}
  t \mapsto \int|u(t)|^{2} \log \(|u(t)|^{2}\) \in L_{\mathrm{loc}}^{\infty}(\mathbb{R}),
\end{equation*}
hence the claim.
\smallbreak

\noindent\textbf{Step 2.} We show that $u \in C\(\mathbb{R},
  W_{1}^s\)$. We check  that the map $t \mapsto \int
F_{2}(u(t)) $ is continuous,  and then we need to show that so is $t
\mapsto \int F_{1}(u(t))$. As in the proof of Lemma~\ref{u in C(H^s)},
we consider continuity at $t=0$ only. Resuming the computation for the preceding paragraph, we derive
\begin{align*}
\frac{|\lambda|}{2} \int\left|F_{1 \eps}\(u_{\eps}(t)\)\right| &
                                                                 =-E_{\eps}\(u_{\eps}(t)\)+\frac{1}{2} \int\left|(-\Delta)^{s/2} u_{\eps}(t)\right|^{2}\\
&\quad +\frac{\lambda}{2}\(\int F_{2 \eps}\(u_{\eps}(t)\)-\int \mu_{\eps}\(\left|u_{\eps}(t)\right|\)\) \\
& \leq -E_{\eps}(\varphi)+\frac{1}{2}
                                                                                                             e^{4|\lambda||t|}\|(-\Delta)^{s/2} \varphi\|_{L^{2}}^{2}\\
&\quad+\frac{\lambda}{2}\(\int F_{2 \eps}\(u_{\eps}(t)\)-\int \mu_{\eps}\(\left|u_{\eps}(t)\right|\)\).
\end{align*}
In view of Fatou's Lemma and Lemma~\ref{lem:mu&F converg}, we infer
\begin{equation*}
  \frac{|\lambda|}{2} \int\left|F_{1}(u(t))\right| \leq -E(\varphi)+\frac{1}{2} e^{4|\lambda||t|}\|(-\Delta)^{s/2} \varphi\|_{L^{2}}^{2}+\frac{\lambda}{2} \int F_{2}(u(t))-\frac{\lambda}{2} \int|u(t)|^{2} .
\end{equation*}
Passing to the limit  $t\to 0$ yields
\begin{align*}
\limsup _{t \to 0} \frac{|\lambda|}{2} \int\left|F_{1}(u(t))\right| & \leq -E(\varphi)+\frac{1}{2}\|(-\Delta)^{s/2} \varphi\|_{L^{2}}^{2}+\frac{\lambda}{2} \int F_{2}(\varphi)-\frac{\lambda}{2} \int|\varphi|^{2} \\
& =-\frac{\lambda}{2} \int F_{1}(\varphi)=\frac{|\lambda|}{2} \int\left|F_{1}(\varphi)\right|.
\end{align*}
Thanks to Fatou's Lemma,
\begin{equation*}
  \int\left|F_{1}(\varphi)\right|\leq\liminf_{t\to 0} |F_1(u(t))|,
\end{equation*}
hence the proposition.
\end{proof}
Since regardless of the sign of $\lambda$, $u\in C(\R,W_1^s)$, arguing
like in  the proof of \cite[Lemma~2.6]{Caz83}, we infer
\begin{equation*}
  i \d_{t} u-(-\Delta)^s u=\lambda u \log {\(|u|^{2}\)} \quad \text {
    in } (W_{1}^s)^*. 
\end{equation*}

\subsection{The $H^1$ case}
\label{sec:H1}

To conclude the proof of  Theorem~\ref{theo:Hs}, we now assume
$\varphi\in H^1(\R^d)$. Since $0<s<1$, we already know that
\eqref{eq:logNLS} has a unique solution $u\in C(\R,H^s(\R^d))$. We
note that the solution $u_\eps$ to \eqref{approxi} is bounded in
$H^1(\R^d)$, uniformly on any time interval $[-T,T]$ and in $\eps\in
(0,1]$. Indeed, applying the gradient to \eqref{approxi} yields
\begin{equation*}
  i\d_t \nabla u_\eps - (-\Delta)^s\nabla u_\eps =2\lambda \nabla
  u_\eps \log\(|u_\eps|+\eps\) +2\lambda
  \frac{u_\eps}{|u_\eps|+\eps}\nabla |u_\eps|, 
\end{equation*}
and the standard $L^2$ estimate readily provides
\begin{equation*}
  \frac{d}{dt}\|\nabla u_\eps\|_{L^2}^2 \le 4|\lambda| \|\nabla
  u_\eps\|_{L^2} \|\nabla |u_\eps|\|_{L^2}\le 4|\lambda| \|\nabla
  u_\eps\|_{L^2}^2.
\end{equation*}
The conclusion of Theorem~\ref{theo:Hs} then follows from the same
arguments as above, when we proved that $u\in C(\R,H^s(\R^d))$.

\section{The Cauchy problem in the $H^{2s}$ regularity}
\label{sec:X}
In this section we show that if $\varphi \in X_\alpha^{2s}= H^{2s}\cap
\F(H^\alpha)$, then  the solution $u\in C(\R,H^s)$ provided by
Theorem~\ref{theo:Hs} actually belongs to $C_w\cap L^\infty_{\rm
  loc}(\R ,X_\alpha^{2s})$
(note the obvious
relation $X_\alpha^{2s}\subset H^s$). 
\smallbreak

The strategy is inspired by the classical one in the case of the
nonlinear Schr\"odinger equation, when $H^2$ regularity is addressed,
see \cite{Kato87} (see also \cite{CazCourant}): we first prove that $\d_t u\in
L^\infty_{\rm loc}(\R,L^2)$, and eventually use the equation,
\eqref{eq:logNLS}, to conclude that $(-\Delta)^su \in L^\infty_{\rm
  loc}(\R,L^2)$. The intermediate step consists in considering the
nonlinearity, to show that 
$  u\log |u|^2\in L^\infty_{\rm  loc}(\R,L^2)$: due to the singularity
of the logarithm at the origin, this is by no means obvious (in
particular, the information $u\in C(\R,H^s)$ and the Sobolev embedding
do not suffice to conclude). The first step is indeed:

\begin{lemma}\label{lem:dtu}
    Let  $\alpha>0$, $\varphi\in X_\alpha^{2s}$, and, for $\eps>0$, $u_\eps$ solve
    \eqref{approxi}. For all $t \in \mathbb{R}$ we have 
\begin{equation*}
  \left\|\d_{t} u_{\eps}(t)\right\|_{L^{2}}^{2} \leq
  e^{4|\lambda t|}\left\|\d_{t} u_{\eps}(0)\right\|_{L^{2}}^{2},
\end{equation*}
and there exists a map $C$ independent of $\eps\in (0,1)$ such that
\begin{equation*}
  \|\d_{t} u_{\eps}(0)\|_{L^{2}}\le C\(
\|\varphi\|_{H^{2s}},\|\<x\>^\alpha\varphi\|_{L^2}\). 
\end{equation*}
\end{lemma}

\begin{proof}
  For the first part of the lemma, we compute
 \begin{align*}
\frac{d}{d t}\left\|\d_{t} u_{\eps}\right\|_{L^{2}}^{2} & =2
    \RE\( \d_{t}^{2}  u_{\eps},  \d_{t} u_{\eps}\) \\ 
& =-2 \IM\(\d_{t}\left\{(-\Delta)^s u_{\eps}+2 \lambda u_{\eps} \log \(\left|u_{\eps}\right|+\eps\)\right\}, \d_{t} u_{\eps}\) \\
& =-4 \lambda \IM\(\frac{u_{\eps}}{\left|u_{\eps}\right|+\eps} \d_{t}\left|u_{\eps}\right|, \d_{t} u_{\eps}\) \leq 4|\lambda|\|\d_t u_\eps (t)\|^2_{L^2},
\end{align*}
hence the announced inequality by Gronwall Lemma. 
Now in view of \eqref{approxi},
\begin{align*}
  \|\d_{t} u_{\eps}(0)\|_{L^{2}}&\le \|(-\Delta)^s
  u_{\eps}(0)\|_{L^{2}}+2|\lambda| \left\|
      u_\eps(0)\log(|u_\eps(0)|+\eps)\right\|_{L^{2}} \\
  &\le \|\varphi\|_{  H^{2s}} + +2|\lambda| \left\| 
  \varphi\log(|\varphi|+\eps)\right\|_{L^{2}} .
\end{align*}
For $\delta>0$,
\begin{align*}
  \left| \varphi\log(|\varphi|+\eps)\right|&\lesssim |\varphi|\(
  (|\varphi|+\eps)^{-\delta} + (|\varphi|+\eps)^{\delta}\) \lesssim
      |\varphi|^{1-\delta} +|\varphi| (|\varphi|^{\delta}+1) ,
\end{align*}
and, provided that $\delta>0$ is sufficiently small (in terms of $s$
and $\alpha$),
\begin{equation*}
  \left\|    |\varphi|^{1-\delta}\right\|_{L^2}\lesssim \left\|
    \<x\>^\alpha \varphi\right\|_{L^2}^{1-\delta} ,\quad \left\| \varphi
      (|\varphi|^{\delta}+1)  \right\|_{L^2}\lesssim
    \|\varphi\|_{H^{2s}}^{1+\delta} + |\varphi\|_{L^2},
  \end{equation*}
  hence the lemma. 
\end{proof}
Combined with \eqref{MT},
\begin{equation}\label{NT}
    N_{T}:=\sup _{\eps \in(0,1)}\(\left\|u_{\eps}\right\|_{C_{T}\(H^{s}\)}+\left\|\d_{t} u_{\eps}\right\|_{C_{T}\(L^{2}\)}\) \leq C\(T,\|\varphi\|_{X_\alpha^{2s}}\).
\end{equation}

The unique solution $u \in C\(\mathbb{R}, H^{s}\(\mathbb{R}^{d}\)\)$
to \eqref{eq:logNLS} was constructed in Section~\ref{sec:Hs}, obtained
as the limit of $u_\eps$ as $\eps\to 0$, and we deduce from \eqref{NT} that
\begin{equation*}
  u \in W_{\mathrm{loc}}^{1, \infty}\(\mathbb{R},
  L^{2}\(\mathbb{R}^{d}\)\), \quad \d_{t} u_{\eps}(t) \rightharpoonup
  \d_{t} u(t) \quad \text { in } L^{2}\(\mathbb{R}^{d}\). 
\end{equation*}
As announced above, the next step consists in showing that
$u\log|u|^2$ belongs to $L^\infty_{\rm loc}(\R,L^2)$. Using the same
estimates as in the proof of Lemma~\ref{lem:dtu}, it suffices to prove
the following result:
\begin{lemma}\label{lem:X}
  Let $0<s<1$, $0<\alpha<2s$ with $\alpha\le 1$, and $\varphi\in
  X_\alpha^{2s}$. Then the solution $u\in C(\R,H^s)$ provided by
  Theorem~\ref{theo:Hs} also belongs to $C_w\cap L^\infty_{\rm
  loc}(\R ,\F(H^\alpha))$.
\end{lemma}
\begin{proof}
  Let $\eps>0$: multiplying \eqref{approxi} by $\<x\>^\alpha$, we find
   \begin{equation*}
        i \d_{t} (\<x\>^\alpha u_\eps)-\<x\>^\alpha(-\Delta)^s  u_\eps=2\lambda \<x\>^\alpha u_\eps \log {\(|u_\eps|+\eps\)} ,
 \end{equation*}
which can be rewritten as
    \begin{equation*}
        i \d_{t} (\<x\>^\alpha u_\eps)-(-\Delta)^s\(\<x\>^\alpha
        u_\eps\)=2\lambda \<x\>^\alpha u_\eps \log {\(|u_\eps|+\eps\)}
        -[(-\Delta)^s,\<x\>^\alpha]u_\eps. 
      \end{equation*}
  Multiplying the above equation by $\<x\>^\alpha \bar {u_\eps}$, integrating over $\mathbb R^d$ and taking the imaginary part, we
obtain, since $(-\Delta)^s$ is self-adjoint,
\begin{equation*}
  \frac{d}{d t}\left\|\<x\>^\alpha u_{\eps}\right\|_{L^{2}}^2
  \le 2 \left\| \<x\>^\alpha u_{\eps}\right\|_{L^{2}}\left\|
  [(-\Delta)^s,  \<x\>^\alpha] u_{\eps}\right\|_{L^{2}} .
\end{equation*}
The last factor is estimated thanks to Lemma~\ref{lem:commutator}:
for $T>0$ and $t\in [-T,T]$,
\begin{equation*}
  \left\|
  [(-\Delta)^s,  \<x\>^\alpha] u_{\eps}(t)\right\|_{L^{2}} \lesssim
\|u_{\eps}(t)\|_{H^s}\lesssim M_T\lesssim N_T.
\end{equation*}
Gronwall Lemma implies that $u_\eps$ is uniformly bounded in
$L^\infty_T \F(H^\alpha)$, 
and the lemma follows by the same arguments as in
Section~\ref{sec:Hs}. 
\end{proof}

As explained above, we conclude that $(-\Delta)^s u\in C_w\cap
L^\infty_{\rm loc}(\R,L^2)$, and Theorem~\ref{theo:X} follows, keeping
Lemma~\ref{lem:X} in mind.

\bibliographystyle{abbrv}
\bibliography{biblio}

\begin{thebibliography}{10}

\bibitem{AnYang2023}
X.~An and X.~Yang.
\newblock Convergence from power-law to logarithm-law in nonlinear fractional
  {S}chr\"{o}dinger equations.
\newblock {\em J. Math. Phys.}, 64(1):Paper No. 011506, 13, 2023.

\bibitem{applebaum2004levy}
D.~Applebaum.
\newblock L{\'e}vy processes-from probability to finance and quantum groups.
\newblock {\em Notices of the AMS}, 51(11):1336--1347, 2004.

\bibitem{Ardila2017}
A.~H. Ardila.
\newblock Existence and stability of standing waves for nonlinear fractional
  {S}chr{\"o}dinger equation with logarithmic nonlinearity.
\newblock {\em Nonlinear Anal., Theory Methods Appl., Ser. A, Theory Methods},
  155:52--64, 2017.

\bibitem{bhattarai2016existence}
S.~Bhattarai.
\newblock Existence and stability of standing waves for nonlinear
  {S}chr\"odinger systems involving the fractional {L}aplacian.
\newblock {\em arXiv preprint arXiv:1604.01718}, 2016.

\bibitem{BiMy76}
I.~Bia{\l}ynicki-Birula and J.~Mycielski.
\newblock Nonlinear wave mechanics.
\newblock {\em Ann. Physics}, 100(1-2):62--93, 1976.

\bibitem{cabre2014nonlinear}
X.~Cabr{\'e} and Y.~Sire.
\newblock Nonlinear equations for fractional {Laplacians}. {I}: {Regularity},
  maximum principles, and {Hamiltonian} estimates.
\newblock {\em Ann. Inst. Henri Poincar{\'e}, Anal. Non Lin{\'e}aire},
  31(1):23--53, 2014.

\bibitem{CaGa18}
R.~Carles and I.~Gallagher.
\newblock Universal dynamics for the defocusing logarithmic {S}chr{\"o}dinger
  equation.
\newblock {\em Duke Math. J.}, 167(9):1761--1801, 2018.

\bibitem{carles2023low}
R.~Carles, M.~Hayashi, and T.~Ozawa.
\newblock Low regularity solutions to the logarithmic {S}chr\"odinger equation.
\newblock {\em Pure Appl. Anal.}, 2024.
\newblock To appear. Archived at \url{https://arxiv.org/abs/2311.01801}.

\bibitem{Caz83}
T.~Cazenave.
\newblock Stable solutions of the logarithmic {S}chr\"odinger equation.
\newblock {\em Nonlinear Anal.}, 7(10):1127--1140, 1983.

\bibitem{CazCourant}
T.~Cazenave.
\newblock {\em Semilinear {S}chr\"odinger equations}, volume~10 of {\em Courant
  Lecture Notes in Mathematics}.
\newblock New York University Courant Institute of Mathematical Sciences, New
  York, 2003.

\bibitem{CaHa80}
T.~Cazenave and A.~Haraux.
\newblock \'{E}quations d'\'evolution avec non lin\'earit\'e logarithmique.
\newblock {\em Ann. Fac. Sci. Toulouse Math. (5)}, 2(1):21--51, 1980.

\bibitem{cho2013cauchy}
Y.~Cho, H.~Hajaiej, G.~Hwang, and T.~Ozawa.
\newblock On the {C}auchy problem of fractional {S}chr{\"o}dinger equation with
  {H}artree type nonlinearity.
\newblock {\em Funkcialaj Ekvacioj}, 56(2):193--224, 2013.

\bibitem{CHHO14}
Y.~Cho, H.~Hajaiej, G.~Hwang, and T.~Ozawa.
\newblock On the orbital stability of fractional {S}chr\"{o}dinger equations.
\newblock {\em Commun. Pure Appl. Anal.}, 13(3):1267--1282, 2014.

\bibitem{COX11}
Y.~Cho, T.~Ozawa, and S.~Xia.
\newblock Remarks on some dispersive estimates.
\newblock {\em Commun. Pure Appl. Anal.}, 10(4):1121--1128, 2011.

\bibitem{d2015fractional}
P.~d'Avenia, M.~Squassina, and M.~Zenari.
\newblock Fractional logarithmic {S}chr{\"o}dinger equations.
\newblock {\em Mathematical Methods in the Applied Sciences},
  38(18):5207--5216, 2015.

\bibitem{david2004levy}
A.~David.
\newblock Levy processes and stochastic calculus.
\newblock {\em Cambridge Studies in Advanced Mathematics, Cambridge University
  Press, Cambridge}, 2004.

\bibitem{Hitch2012}
E.~Di~Nezza, G.~Palatucci, and E.~Valdinoci.
\newblock Hitchhiker's guide to the fractional {S}obolev spaces.
\newblock {\em Bull. Sci. Math.}, 136(5):521--573, 2012.

\bibitem{Dinh2018}
V.~D. Dinh.
\newblock Well-posedness of nonlinear fractional {S}chr\"odinger and wave
  equations in {S}obolev spaces.
\newblock {\em Int. J. Appl. Math.}, 31:483---525, 2018.

\bibitem{guo2008existence}
B.~Guo, Y.~Han, and J.~Xin.
\newblock Existence of the global smooth solution to the period boundary value
  problem of fractional nonlinear {S}chr{\"o}dinger equation.
\newblock {\em Applied Mathematics and Computation}, 204(1):468--477, 2008.

\bibitem{guo2012existence}
B.~Guo and D.~Huang.
\newblock Existence and stability of standing waves for nonlinear fractional
  {S}chr{\"o}dinger equations.
\newblock {\em Journal of Mathematical Physics}, 53(8), 2012.

\bibitem{guo2010global}
B.~Guo and Z.~Huo.
\newblock Global well-posedness for the fractional nonlinear {S}chr{\"o}dinger
  equation.
\newblock {\em Communications in Partial Differential Equations},
  36(2):247--255, 2010.

\bibitem{guo2006some}
X.~Guo and M.~Xu.
\newblock Some physical applications of fractional {S}chr{\"o}dinger equation.
\newblock {\em Journal of mathematical physics}, 47(8), 2006.

\bibitem{GuoWang2014}
Z.~Guo and Y.~Wang.
\newblock Improved {S}trichartz estimates for a class of dispersive equations
  in the radial case and their applications to nonlinear {S}chr\"{o}dinger and
  wave equations.
\newblock {\em J. Anal. Math.}, 124:1--38, 2014.

\bibitem{HO}
M.~Hayashi and T.~Ozawa.
\newblock The {C}auchy problem for the logarithmic {S}chr\"odinger equation
  revisited.
\newblock Preprint,
  \href{https://arxiv.org/abs/2309.01695}{https://arxiv.org/abs/2309.01695},
  2023.

\bibitem{hefter1985application}
E.~F. Hefter.
\newblock Application of the nonlinear {S}chr{\"o}dinger equation with a
  logarithmic inhomogeneous term to nuclear physics.
\newblock {\em Physical Review A}, 32(2):1201, 1985.

\bibitem{HongSire2015}
Y.~Hong and Y.~Sire.
\newblock On fractional {S}chr\"{o}dinger equations in {S}obolev spaces.
\newblock {\em Commun. Pure Appl. Anal.}, 14(6):2265--2282, 2015.

\bibitem{Kato87}
T.~Kato.
\newblock On nonlinear {S}chr\"odinger equations.
\newblock {\em Ann. IHP (Phys. Th\'eor.)}, 46(1):113--129, 1987.

\bibitem{Laskin2000}
N.~Laskin.
\newblock Fractional quantum mechanics and {L}\'{e}vy path integrals.
\newblock {\em Phys. Lett. A}, 268(4-6):298--305, 2000.

\bibitem{Laskin2002}
N.~Laskin.
\newblock Fractional {S}chr\"{o}dinger equation.
\newblock {\em Phys. Rev. E (3)}, 66(5):056108, 7, 2002.

\bibitem{Li19}
D.~Li.
\newblock On {K}ato-{P}once and fractional {L}eibniz.
\newblock {\em Rev. Mat. Iberoam.}, 35(1):23--100, 2019.

\bibitem{WangZhang2019}
Z.-Q. Wang and C.~Zhang.
\newblock Convergence from power-law to logarithm-law in nonlinear scalar field
  equations.
\newblock {\em Arch. Ration. Mech. Anal.}, 231(1):45--61, 2019.

\bibitem{ZHANG2018}
H.~Zhang and Q.~Hu.
\newblock Existence of the global solution for fractional logarithmic
  {Schr{\"o}dinger} equation.
\newblock {\em Comput. Math. Appl.}, 75(1):161--169, 2018.

\bibitem{Zlo10}
K.~G. Zloshchastiev.
\newblock Logarithmic nonlinearity in theories of quantum gravity: {O}rigin of
  time and observational consequences.
\newblock {\em Grav. Cosmol.}, 16:288--297, 2010.

\end{thebibliography}
\end{document}